\documentclass[9pt]{article}
\usepackage{amsfonts}
\usepackage{mathrsfs}
\usepackage{amsmath}
\usepackage{amssymb}
\usepackage{amssymb,amsfonts,amsthm,amsbsy,amsmath,pifont,fleqn,txfonts}
\usepackage[mathscr]{eucal}
\usepackage{mathrsfs}

\newtheorem{theorem}{Theorem}
\newtheorem{lemma}{Lemma}
\theoremstyle{definition}
\newtheorem{definition}{Definition}
\newtheorem{remark}{Remark}

\theoremstyle{corollary}

\newcommand{\De}{\Delta}

\newcommand{\Om}{{\Omega}}

\newcommand{\al}{{\alpha}}

\newcommand{\ga}{{\gamma}}

\newcommand{\si}{{\sigma}}

\newcommand{\sC}{\mathscr{C}}
\newcommand{\sD}{\mathscr{D}}
\newcommand{\sE}{\mathscr{E}}

\newcommand{\N}{{\mathbb N}}

\newcommand{\R}{{\mathbb R}}

\newcommand{\cF}{{  F}}

\newcommand{\cW}{{  W}}

\newcommand{\g}{{\nabla}}

\newcommand{\pd}{\partial}

\newcommand{\di}{{\rm div\, }}
\newcommand{\wrt}{{with respect to }}

\begin{document}
\title{Pullback attractors on time-dependent spaces for fluid-structure interaction systems.}
\author{Tamara Fastovska\footnote{\small e-mail:
 fastovskaya@karazin.ua} \\Kharkiv Karazin National University, \\  4 Svobody sq., 61077  Kharkiv, Ukraine\\Kharkiv Automobile and Highway National University,  \\25 Yaroslava Mudrogo st., 61002  Kharkiv, Ukraine}

\maketitle

 \centerline{(Communicated by the associate editor name)}
\begin{abstract}
We study the long-time dynamics of a non-autonomous coupled system consisting of the 3D linearized
Na\-vier--Stokes equations  and
nonlinear elasticity equations.
We show that this problem generates a process on time-dependent spaces possessing a pullback attractor.
\end{abstract}
\section{Introduction}
We consider a coupled non-autonomous system with time-dependent coefficients which describes
interaction of a homogeneous viscous fluid
which  occupies a domain $  \EuScript O$ bounded by
the solid walls of the container $S$ and a horizontal boundary $\Om$
on which a thin nonlinear elastic plate is placed.
The motion of the fluid is described
by  linearized 3D Navier--Stokes equations.
To describe  deformations of the plate
 we consider  a generalized plate model  which accounts only for
transversal displacements and
 covers
a general large deflection
Karman type model (see, e.g., \cite{lagnese}).  However, our results can be also
applied in the cases of nonlinear Berger and Kirchhoff plates.

\smallskip\par

Our mathematical model is formulated as  follows.
\par
 Let $  \EuScript O\subset \R^3$
 be a bounded domain  with a sufficiently smooth
 boundary $\partial  \EuScript O$. We assume that
$\partial  \EuScript O=\overline{\Omega}\cup \overline{S}$,
 where $\Om\cap S=\emptyset$ and
$$
\Om\subset\{ x=(x_1;x_2;0)\, :\,x'\equiv
(x_1;x_2)\in\R^2\}
$$ with the smooth contour $\Gamma=\partial\Om$
and $S$
 is a  surface which lies in the subspace $\R^3_- =\{ x_3\le 0\}$.
 The exterior normal on $\partial  \EuScript O$ is denoted
 by $n$. We have that $n=(0;0;1)$ on $\Om$.
 We consider the following  Navier--Stokes equations in $  \EuScript O$
for the fluid velocity field $v=v(x,t)=(v^1(x,t);v^2(x,t);v^3(x,t))$
and for the pressure $p(x,t)$:
\begin{equation}\label{fl.1}
  \mu(t) v_t-\Delta v+\nabla p=f(x,t)\quad {\rm in\quad}   \EuScript O
   \times(\tau,+\infty),\;\tau\in \mathbb R.
\end{equation}
   \begin{equation}\label{fl.2}
   \di v=0 \quad {\rm in}\quad   \EuScript O
   \times(\tau,+\infty),\;\tau\in \mathbb R.
  \end{equation}
where $f(x, t)$ is a volume force.\par
   We supplement (\ref{fl.1}) and (\ref{fl.2}) with  the non-slip  boundary
   conditions imposed  on the velocity field $v=v(x,t)$:
\begin{equation}\label{fl.4}
v=0 ~~ {\rm on}~S;
\quad
v\equiv(v^1;v^2;v^3)=(0;0;u_t) ~~{\rm on} ~ \Om.
\end{equation}
Here $u=u(x,t)$ is the transversal displacement
of the plate occupying $\Om$ and satisfying
the following  equation:
\begin{equation*}
\rho(t) u_{tt} + \De^2 u + \cF(u)=g(x,t)+p|_\Om,\;
~~{\rm in}~~ \Omega \times (\tau, \infty),\;\tau\in \mathbb R.
\end{equation*}
where $g(x,t)$ is a given body force on the plate, $\cF(u)$ is a nonlinear feedback force which will be specified later. \par
We impose clamped boundary conditions on the plate
\begin{equation}
u|_{\pd\Om}=\left.\frac{\pd u}{\pd n} \right|_{\pd\Om}=0 \label{plBC}
\end{equation}
and supply \eqref{fl.1}--\eqref{plBC} with initial data of the form
\begin{equation}
 v(x,\tau)=v_\tau,\quad u(x,\tau)=u_\tau^0, \quad u_t(x,\tau)=u_\tau^1, \label{IC}
\end{equation}
We  note that  \eqref{fl.2} and \eqref{fl.4} imply the following
compatibility condition
\begin{equation}\label{Com-con}
\int_\Om u_t(x',t) dx'=0 \quad \mbox{for all}~~ t\ge \tau\;\tau\in \mathbb R.
\end{equation}
This condition fulfills when
\[
\int_\Om u(x',t) dx'=const \quad \mbox{for all}~~ t\ge \tau,\;\tau\in \mathbb R.
\]
which can be interpreted as preservation of the volume of the fluid.
\par

\medskip \par
In this paper our main point of interest is well-posedness and
long-time dynamics of
solutions to the coupled problem in
\eqref{fl.1}--\eqref{IC}
for the velocity $v$ and the displacement $u$. \par
We consider  this problem under
rather general hypotheses concerning nonlinearity.
These hypotheses cover the cases of von Karman, Berger and Kirchhoff
plates. We show that  problem \eqref{fl.1}--\eqref{IC} generates
a process on a family of time-dependent energy spaces.
Our main result
states that under some natural conditions concerning
 feedback forces system \eqref{fl.1}--\eqref{IC}
possesses a pullback attractor.
To establish this results we  adjust the compensated compactness  approach widely used for dynamical systems of autonomous equations  (see \cite{cl-jdde}, \cite{cl-mem} and
\cite[Chapters 7,8]{cl-book} and also the references therein) to processes and pullback attractors.
\par
The mathematical studies of the long-time behavior of autonomous problems
of fluid--structure interaction
in the case of viscous fluids and elastic plates/bodies
have a long history (see, e.g.  \cite{Chu_2010, CF, ChuRyz2011, F} and references  therein).\par
In the present work we investigate the existence of a pullback attractor in case of time-dependent coefficients in the main parts of equations. For such problems only a few  papers are devoted to the existence of pullback attractors \cite{Carab, Conti, Duane}. In paper \cite{Carab}  a wave equation with time-dependent coefficient before the damping term is considered, consequently the energy space does not depend on the time parameter. Papers \cite{Conti, Duane} deal with wave equations with time-dependent coefficients before the second derivatives with respect to time and to the space variable respectively. In both works the existence of time-dependend pullback attractor in a scale of spaces is established.   \par
In our paper we consider for the first time, to the best of our knowledge, the    interaction model for  a  Newtonian fluid and a plate  with time-dependent coefficients before the time deriatives. The peculiarity of the problem considered consists in the absence of any mechanical damping  in the  plate component and the strong coupling of fluid and plate components.   \par
    We prove the well-posedness of the system considered and investigate the
long-time dynamics of
solutions to the coupled problem in
\eqref{fl.1}--\eqref{IC}. In order to show the existence of a time-dependent pullback attractor we derive an abstract result on the asymtotic  compactness of processess on time-dependent spaces.   
\par
The paper is organized as follows.
 In Section 2 we introduce   notations, recall
some properties of Sobolev type spaces with non-integer
indexes   on bounded domains and collect some regularity
properties of (stationary)
Stokes problem which we use in the further considerations. The main notions from the theory of pullback attractors and new abstract results are presented in Section 3. 
Our main result in  Section 4
is Theorem~\ref{WP} on  well-posedness and existence of time-dependend absorbing set.
Our main result in  Section 5 states existence of a
pullback
attractor.
The  argument is based on the
property established in Theorem~\ref{pr:com}.

\section{Spaces and notations.} 
Now we introduce Sobolev type spaces which are used in what follows
(see e.g. \cite{Triebel78}).\par
Let $D$ be a sufficiently smooth domain  and $s\in\R$.
We denote  by $H^s(D)$ the Sobolev space of order $s$
on a set $D$ which we define as restriction (in the sense of distributions)
 of the
space $H^s(\R^d)$ (introduced via Fourier transform).
We denote by $\|\cdot \|_{s,D}$ the norm in  $H^s(D)$
which we define by the relation
\[
\|u\|_{s,D}^2=\inf\left\{\|w\|_{s,\R^d}^2\, :\; w\in H^s(\R^d),~~ w=u ~~
\rm{on}~~D
    \right\}
\]
We also use the notation $\|\cdot \|_{D}=\|\cdot \|_{0,D}$
for the corresponding $L_2$-norm and, similarly, $(\cdot,\cdot)_D$ for the $L_2$
inner product.
We denote by $H^s_0(D)$ the closure of $C_0^\infty(D)$ in  $H^s(D)$
(\wrt  $\|\cdot \|_{s,D}$) and introduce the spaces
\[
H^s_*(D):=\left\{f\big|_D\, :\;  f\in H^s(\R^d), \;
{\rm supp }\, f\subset \overline{D}\right\},\quad s\in \R.
\]
Since the extension by zero of elements from $H^s_*(D)$ gives us an
element of $H^s(\R^d)$,
these spaces $H^s_*(D)$ can be treated not only as functional spaces defined
on $D$ (and contained in $H^s(D)$) but also as  (closed) subspaces
of $ H^s(\R^d)$. Below we need them to describe
boundary traces on $\Om\subset\partial   \EuScript O$.
We endow the classes $H^s_*(D)$ with the induced norms
 $\|f \|^*_{s,D}= \| f \|_{s,\R^d}$
for $f\in H^s_*(D)$. It is clear that
\[
\|f \|_{s,D}\le \|f \|^*_{s,D}, ~~ f\in H^s_*(D).
\]
It is known  (see \cite[Theorem 4.3.2/1]{Triebel78})
that $C_0^\infty(D)$ is dense in $H^s_*(D)$ and
\begin{align*}
& H^s_*(D)\subset H^s_0(D)\subset H^s(D),~~~ s\in\R;\\
& H^s_0(D) =  H^s(D),~~~ -\infty< s\le 1/2;\\
& H^s_*(D)= H^s_0(D),~~~ -1/2<  s<\infty,~~ s-1/2\not\in
\{ 0,1,2,\ldots\}.
\end{align*}
In particular, $H^s_*(D)= H^s_0(D)= H^s(D)$ for $|s|<1/2$. By
 \cite[Remark 4.3.2/2]{Triebel78} we also have that
 $H^s_*(D)\neq  H^s(D)$ for
 $|s|>1/2$. Note that  in the notations of \cite{LiMa_1968}
the space $H^{m+1/2}_*(D)$ is the same as $H^{m+1/2}_{00}(D)$ for every
 $m= 0,1,2,\ldots$ , and for $s=m+\si$ with $0<\si<1$ we have
\[
\|u \|^*_{s,D}=\left\{ \| u\|^2_{s,D}
+\sum_{|\al|=m}\int_D \,\frac{|D^\al u(x)|^2}{d(x,\pd D)^{2\si}}\, dx
\right\}^{1/2},
\]
where $d(x,\pd D)$ is the distance between $x$ and $\pd D$.
The norm  $\|\cdot \|^*_{s,D}$ is equivalent to
 $\|\cdot \|_{s,D}$ in the case when
$s>-1/2$ and  $s-1/2\not\in\{ 0,1,2,\ldots\}$,
but  not equivalent in general.
\par
Understanding adjoint spaces \wrt duality between
$C_0^\infty(D)$ and $[C_0^\infty(D)]'$
by Theorems 4.8.1 and 4.8.2 from \cite{Triebel78} we also have that
\begin{align*}
 [H^s_*(D)]'= H^{-s}(D),~ s\in\R, ~~~\mbox{and} ~~~
 [H^s(D)]' =  H_*^{-s} (D),~ s\in (-\infty,1/2).
\end{align*}
To describe fluid velocity fields
we introduce the following scale of spaces.
\par
Let $\mathscr{C}(  \EuScript O)$  be the class of
$C^\infty$ vector-valued solenoidal (i.e., divergence-free) functions
$ v=(v^1;v^2;v^3)$
on $\overline{\EuScript O}$ which vanish in a neighborhood of $S$
and such that $v^1=v^2=0$ on $\Om$.
We denote by $X_t$ the closure of $\sC(  \EuScript O)$ \wrt  the following $L_2$-norms 
$$
\|\cdot\|_{X_t}=\mu(t)\|\cdot\|_{L_2(  \EuScript O)}
$$
 and by $Y$ the closure  \wrt the $H^1(  \EuScript O)$-norm. One
can see that
\[
X_t=\left\{ v=(v^1;v^2;v^3)\in [L_2(  \EuScript O)]^3\, :\; {\rm div}\, v=0;\;
\gamma_n v\equiv (v,n)=0~\mbox{on}~ S,\; t\in\mathbb R\right\}
\]
and
\[
Y=\left\{
v=(v^1;v^2;v^3)\in [H^1(  \EuScript O)]^3\, \left| \begin{array}{l}
 {\rm div}\, v=0,\;
v=0~\mbox{on}~ S, \\ v^1=v^2=0~\mbox{on}~\Om \end{array} \right.
  \right\}.
\]
The space $Y$ is endowed with the norm $\|\cdot\|_Y= \|\nabla\cdot\|_{L_2(  \EuScript O)}$.
For some details concerning this type spaces we refer to \cite{temam-NS},
for instance.\par
We also need the Sobolev spaces consisting of functions with zero average
on the domain $\Om$, namely
we consider the space
\[
\widehat{L}_2(\Om)=\left\{u\in L_2(\Om): \int_\Om u(x') dx' =0 \right\}
\]
and also the scale of time-dependent spaces 
\[\widehat{L}_t^2(\Omega)=\left\{ L^2(\Omega): \|\cdot\|_{\widehat{L}_t^2(\Omega)}^2=\rho(t) \|\cdot\|_{{L}_2}^2,\; t\in\mathbb R\right\}.\]
 We use the notation $\widehat H^s(\Om)=H^s(\Om)\cap\widehat L_2(\Om)$ for $s>0$
with the standard $H^s(\Om)$-norm.
The notations  $\widehat H^s_*(\Om)$ and $\widehat H^s_0(\Om)$
have a similar meaning.

\section{Abstract results on  attractors.}
We begin with some definitions from the theory of processses.
\begin{definition}
\label{d1}
A two paramter family $U(t,\tau): H_\tau\to H_t$, $t\ge \tau\in \mathbb R$ of operators in a scale of Banach spaces $H_t$, $t\in \mathbb R$ is called a process if
\begin{itemize}
\item  $U(\tau, \tau)=I$
\item $U(t, s)U(s, \tau)=U(t, \tau),\; t\ge s\ge \tau\in \mathbb R$
\end{itemize}
\end{definition}

\begin{definition}
\label{d2}
The family of sets $\EuScript B=\{B_t\}_{t\in\mathbb R}$, for $B_t\in H_t$ is positively invariant  if $U(t,\tau)B_\tau\subset B_t$  $\forall t\in\mathbb R$.
\end{definition}

\begin{definition} 
\label{d3}
The family of bounded sets $\EuScript B=\{B_t\}_{t\in\mathbb R}$, for $B_t\in H_t$  is uniformly bounded if there exists $R>0$ such that $B_t\in \mathbb B_t(R)=\{z\in H_t, \|z\|_{H_t}\le R \}$ for any $ t\in\mathbb R$.
\end{definition}
\begin{definition} 
To study the asymptotic
behavior of the operators $U(t, \tau)$ we need to define a suitable object which attracts solutions of the system originating sufficiently far in the past. In order to do it we need to introduce the notion of absorbtion and attraction.
\end{definition}
\begin{definition}
\label{d4}
The family of uniformly  bounded  sets $\EuScript B=\{B_t\}_{t\in\mathbb R}$, is  time-dependent absorbing  if for any $R>0$ there exists $ \Theta=\Theta(R)$ such that  $U(t,\tau)\mathbb B_\tau(R)\in B_t$ for any   $\tau\le t-\Theta$.
\end{definition}
The process $U(t, \tau)$ is called dissipative whenever it admits a pullback absorbing family.
\begin{definition}
\label{d31}
A time dependent $\omega$-limit  of any pullback absorbing family $\EuScript B=\{B_t\}_{t\in\mathbb R}$, for $B_t\in H_t$  is  the family $\Omega=\{\omega_t(\EuScript B)\}_{t\in\mathbb R}$, where
\begin{equation}
\omega_t(\EuScript B)=\bigcap\limits_{y\le t}\overline{\bigcup\limits_{\tau\le y}U(t, \tau)B_\tau}.
\end{equation}
\end{definition}
\begin{definition}
\label{d5}
The family of uniformly  bounded  sets $\EuScript K=\{K_t\}_{t\in\mathbb R}$ is pullback attracting if for every uniformly bounded family $ \EuScript B=\{B_t\}_{t\in\mathbb R}$  
$$\lim\limits_{\tau\to-\infty}\delta_t(U(t,\tau)B_\tau, K_t)=0,$$
 where $\delta_t(B,C)=\sup\limits_{x\in B}\inf\limits_{y\in C}\|x-y\|_{H_t}$ denotes the Hausdorff semidistance.
\end{definition}
Now we are in position to define the pullback attractor.
\begin{definition}
\label{d6}
A process is asymptotically compact if there exists a pullback attracting family of compact  sets $\EuScript K=\{K_t\}_{t\in\mathbb R}$,  $K_t\in H_t$.
\end{definition}
\begin{definition}
\label{d7}
Pullback attractor is the smallest element of pullback attracting families  $\mathbb K=\{\EuScript K=\{K_t\}_{t\in\mathbb R}\}$, where  $K_t\subset H_t$ are compact in the corresponding spaces.
\end{definition}
The classical approach (see, e.g. \cite{}) to verification of asymptotic compactness of a process consists in finding a decomposition $U(t, \tau)=U_0(t, \tau)+U_1(t, \tau)$ with the properties 
$$\|U_0(t, \tau)z\|_{H_t}\le Ce^{-\delta(t-\tau)},\;C,\delta>0,\;z\in H_\tau$$ and 
$$\sup\limits_{t\ge \tau}\|U_1(t, \tau)z\|_{R_t}\le M,$$
where $R_t$ is a compactly embedded into $H_t$ in Banach space. However, for the system considered it is not obvious how to get such a decomposition due to strong coupling (fluid and plate components cannot be splitted in terms of constraction of Galerkin approximations). Therefore, we need to derive another criterion for   asymptotic compactness. In order to do this we adjust to our situation the method of compensated compactness.  
\begin{theorem}
\label{pr:com}
Let $\EuScript D=\{D_t\}_{t\in\mathbb R}$ be a time-dependent  absorbing family of a process $U(t,\tau):H_\tau\to H_t$ and for any $\varepsilon>0$  there exists $T_0=T_0(\varepsilon)>0$ such that for any $y_1, y_2\in D_{t-T_0}$
\begin{equation}
\label{ac}
\|U(t,t-T_0)y_1-U(t,t-T_0)y_2\|_{H_t}\le \varepsilon+\Phi_{T_0, t}(y_1, y_2), 
\end{equation}
where the function  $\Phi_{T_0, t}(y_1, y_2): D_{t-T_0}\times D_{t-T_0}\to \mathbb R$ possesses the property 
\begin{equation}
\label{p}\liminf\limits_{n\to\infty}\liminf\limits_{m\to\infty}\Phi_{T_0, t}(y_n, y_m)=0
\end{equation}
for any sequence $\{y_n\}\in D_{t-T_0}$.
\end{theorem}
\begin{proof}
We can assume without loss of generality that $\EuScript D$ is positively invariant. Otherwise, we can substitute $D_t$ with $\bigcup\limits_{\tau\le t-\Theta}U(t,\tau)D_\tau\subset D_t$ .\par
We fix $T>0$.  Obviously, we have a representation 
\[\omega_t(\EuScript D)=\bigcap\limits_{k\in\mathbb N}C_k^t,\]
where
\[C_k^t=\overline{U(t, t-kT)D_{t-kT}}.\]   
Now we need to check that 
\begin{equation}
\label{inkl}
C_{k+1}^t\subset C_k^t.
\end{equation}
Indeed, due to the invariance of the family $\EuScript D$ we have $U(t-kT, t-(k+1)T)D_{t-(k+1)T}\subset D_{t-kT}$, consequently,
\begin{multline*}C_{k+1}^t=\overline{U(t, t-(k+1)T)D_{t-(k+1)T}}\\=\overline{U(t, t-kT)U(t-kT, t-(k+1)T)D_{t-(k+1)T}}\subset  \overline{U(t, t-kT)D_{t-kT}}=C_{k}^t.
\end{multline*}
Therefore, we have a sequence of nonempty closed sets 
$$
C_1^t\supset C_2^t\supset...\supset C_{k}^t\supset C_{k+1}^t \supset...
$$ 

To show that $\omega_t(\EuScript D)$ is nonemty and compact it remains to prove that 
\begin{equation}
\label{alphalim} 
\lim\limits_{k\to\infty}\alpha(C_k^t)=0.
\end{equation}
Due to \eqref{inkl} 
\[\alpha(C_k^t)=\alpha(C_k^t\cup C_{k+1}^t)=\max\left(\alpha(C_k^t), \alpha(C_{k+1}^t)\right),\]
consequently, 
\begin{equation}\label{ainkl}
\alpha(C_k^t)\ge \alpha(C_{k+1}^t) 
\end{equation}
for any $k\in\mathbb N$. It follows readily from \eqref{ainkl} that to show \eqref{alphalim} 
 it is sufficient to prove that  for any $\varepsilon>0$ there exists $k_0\in\mathbb N$ such that $\alpha(C_{k_0}^t)\le \varepsilon$.\par
 Now we use the contradiction argument.  Let there exists $\varepsilon_0>0$ such that for any $k\in \mathbb N$ 
\begin{equation}
\label{ep} 
\alpha(C_k^t)>6\varepsilon_0.
\end{equation}
For this $\varepsilon_0$ we choose $T_0=T_0(\varepsilon_0)$ such that \eqref{ac},  \eqref{p} hold. There exist $k_0\in\mathbb N$ and $0<\delta_0<T$ such that $T_0=k_0T-\delta_0$. We use the notation $\EuScript L_0=U(t, t-T_0)D_{t-T_0}=U(t, t-k_0T+\delta_0)D_{t-k_0T+\delta_0}$. Then,
\begin{multline}
\label{l}
C_{k_0}^t=U(t, t-k_0T)D_{t-k_0T}=U(t, t-k_0T+\delta_0)U(t-k_0T+\delta_0, t-k_0T)D_{t-k_0T}\\\subset U(t, t-T_0) D_{t-T_0}=\EuScript L_0.
\end{multline}
It follows from \eqref{ep} and \eqref{l} that 
\[\alpha(\EuScript L_0)\ge \alpha(C_{k_0}^t)>6\varepsilon_0.\]
This implies that there exists a sequence $\{y_n\}_{n=1}^\infty\in D_{t-T_0}$ such that for any $n,m\in \mathbb N$ such that $n\ne m$
\[2\varepsilon_0\le\|U(t,t-T_0)y_n-U(t,t-T_0)y_m\|_{H_t}\le \varepsilon_0+\Phi_{T_0, t}(y_n, y_m),\]
\end{proof}
and, therefore, 
\[\Phi_{T_0, t}(y_n, y_m)\ge \varepsilon_0,\]
which contradicts to \eqref{p}. This means that $\Omega=\{\omega_t(\EuScript D)\}_{t\in\mathbb R}$ is a pullback attracting family of compact sets.
 \section{Well-posedness and existense of absorbing set.}
In this section we prove the existence and uniqueness of solutions to the problem considered, generation of a continuous process, and existence of its time-dependent absorbing set. 
We introduce the scale of phase spaces
\[H_t=X_t\times \widehat {H_0^2(\Omega)}\times \widehat{L_t^2(\Omega)}\]
equipped with the norm
\[\|W\|_{H_t}^2=\mu(t)\|v\|_{L^2(\EuScript O)}^2+\rho(t)\|u\|_{L^2(\Omega)}^2+\|u_t\|_{L^2(\Omega)}^2,\;\;W=(v,u,u_t).\]
Now we impose assumptions on the parameters of problem \eqref{fl.1}--\eqref{IC} (cf. \cite{Conti, Ma}).\par
{\bf Assumptions on $\mu$ and $\rho$.}\par
\begin{enumerate}
\item[(A1)] $\mu(t), \;\rho(t)>0$.
\item[(A2)] $\mu(t), \rho(t) \in C^1(\mathbb R)$ are decreasing functions. 
\item[(A3)] There exists $L>0$ such that
 \[\sup\limits_{t\in \mathbb R}(|\mu(t)|+|\mu'(t)|+|\rho(t)|+|\rho'(t)|)\le L.\] 
\item[(A4)] $\lim\limits_{t\to +\infty}\mu(t)=0$, $\lim\limits_{t\to +\infty}\rho(t)=0$.
\end{enumerate}
{\bf Assumptions on  $\cF$.}\par
\begin{enumerate}
\item[(F1)] There exists $\epsilon>0$ such that $F$ is locally Lipschitz from $H_0^{2-\epsilon}(\Om)$ into $H^{-1/2}(\Om)$, i.e.
\[\| F(u_1)-F(u_2)\|_{-1/2,\Omega}\le  C_R\|u_1-u_2\|_{2-\epsilon, \Omega},\]
for any  $u_1, u_2\in H_0^2(\Omega)$ possessing the property $\|u_i\|_{2,\Omega}\le R, i=1,2$.   
\item[(F2)] There exists a $ C^1$ - functional $\Pi(u)$ on $H_0^2(\Omega)$ such that $F(u)=\Pi'(u)$ and
$\Pi(u)\le Q(\|u\|_{2,\Omega})$, where the fuction  $Q$ is increasing. 
\item[(F3)] There exist $0<\nu<1$ and  $C\ge 0$ such that
 \[(1-\nu)\|\Delta u\|_\Om^2+\Pi(u)+C\ge 0\]
for any $u\in H_0^2(\Omega)$.
\item[(F4)]  There exist $a_1, a_2\ge 0$ and $0<\nu<1$ such that
 \[(F(u), u)\ge a_1\Pi(u)-a_2-(1-\nu)\|\Delta u\|_\Omega^2.\] 
\end{enumerate}
{\bf Assumptions on  $f$ and $g$.}\par
\begin{enumerate}
\item[(G1)]  $ f\in L_{\it{loc}}^2(\mathbb R;Y')$, $ g\in L_{\it{loc}}^2(\mathbb R;H^{-1/2}(\Omega))$.
\item[(G2)]  There exist $\sigma_0,\; C_{f,g}>0$, such that for any  $t\in \mathbb R$ and $\sigma\in [0,\sigma_0]$
 \[\int\limits_{-\infty}^t e^{-\sigma(t-s)}\left(\|f(s)\|_{Y'}^2+\|g(s)\|_{-1/2,\Omega}^2\right)ds\le C_{f, g}.\]
\end{enumerate}
\begin{remark}
We note that assumption (A4) is ipmosed in order to consider the problem in time-dependent spaces. Otherwise, due to assumption (A2) we obtain the equivalence of the norms
\[\|W\|_{H_t}^2 \le \|W\|_{H_\tau}^2\le \max\left\{1,\frac{\mu(\tau)}{\rho(t)} , \frac{\rho(\tau)}{\rho(t)}\right\}\|W\|_{H_t}^2.\]
\end{remark}
\begin{remark}
The examples of function satisfying assumptions (G1), (G2) are periodic functions or $e^{-\kappa t},\;\;\kappa>0$.
\end{remark}
We define the spaces of test functions
$$
\EuScript L_T=\left\{\psi=(\phi, b): \left|\begin{array}{l}
\phi\in L^2(\tau,T;[H^1(\EuScript O)]^3), \phi_t\in L^2(\tau,T;[L_t^2(\EuScript O)]^3,\\
div \phi=0,\;\phi|_S=0,\;\phi|_\Omega=(0,0,b),\\
b\in L^2(\tau,T,\widehat{H_0^2(\Omega)}), \; b_t\in L^2(\tau,T,\widehat{L_t^2(\Omega)})  
\end{array}\right.\right\}
$$
and
$\EuScript L_T^0=\{\psi\in \EuScript L_T\, :\, \psi(t)=0\}.$
\par
In order to make our statements precise we need to introduce the definition of weak solutions  to problem \eqref{fl.1}--\eqref{IC}.
\begin{definition}
A pair of functions $(v(t),u(t))$ is said to be a weak solution to
problem \eqref{fl.1}--\eqref{IC} on a time interval $[\tau,t]$ if
 \begin{itemize}
 \item  $W(t)=(v(t), u(t), u_t(t))\in  L_\infty(\tau,T;H_t);$
\item $v\in  L_2(\tau,T;Y)$, $u_t\in L_2(\tau,T;[H_*^{1/2}(\Omega)]^2)$
\item $ u(\tau)=u_\tau^0;$
 \item  For almost all $ t\in [\tau,T]$
\begin{equation}
 \label{com}
v(t)|_{\Omega}=u_t(t);
\end{equation}
\item For every $\psi=(\phi, b)\in
\EuScript L_T^0$  the following equality holds
\begin{multline}
\label{sol_def}
-\int\limits_\tau^T\mu(t) (v, \phi_t)_\EuScript O dt-\frac12\int\limits_\tau^T\mu'(t) (v, \phi)_\EuScript O  dt+\int\limits_\tau^T\mu(t) (\nabla v, \nabla \phi)_\EuScript O  dt\\
-\int\limits_\tau^T\rho(t) (u_t, b_t)_\Omega dt-\frac12\int\limits_\tau^T\rho'(t) (u_t, b)_\Omega  dt+\int\limits_\tau^T (\nabla u, \nabla b)_\Omega  dt\\
=\int\limits_\tau^T (f(t), \phi)_\EuScript O dt-\int\limits_\tau^T (g(t), \phi)_\Omega  dt+\mu(\tau) ( v_\tau,  \phi(\tau))_\EuScript O\\
+\rho(\tau) (u_\tau^1, b(\tau))_\Omega.
\end{multline}
\end{itemize}
\end{definition}
The following theorem holds true
\begin{theorem}
\label{WP}
Under assumptions (A1)-(A4),  (F1)-(F3),  (G1)-(G2) problem \eqref{fl.1}--\eqref{IC} generates a strongly continuous process $U(t, \tau ) : H_\tau \to H_t$,
$t \ge \tau \in \mathbb R$, satisfying the following continuous dependence property: for every pair of
initial data $W_\tau^i = (v_\tau^i, u_\tau^{0i}, u_\tau^{1i}) \in H_\tau$ such that $\|W_\tau^i\|_{H_\tau} \le R, i = 1, 2$, $R>0$ the difference of the
corresponding solutions satisfies
\begin{equation}
\label{contin}
 \|U(t,  \tau)W_\tau^1 - U(t,  \tau)W_\tau^2\|_{H_t} \le e^{K(t-\tau)}\|W_\tau^1 - W_\tau^2\|_{H_\tau} , \;t \ge\tau,
\end{equation}
for some constant $K = K(R) \ge 0$.\par
The energy  equality
\begin{multline}\label{energy}
\sE(v(t), u(t), u_t(t))+ \int_\tau^t \|\g v\|^2 ds- \frac12\int_\tau^t \mu'(s) \| v\|^2 ds -\frac12 \int_\tau^t \rho'(s) \| u_s\|^2 ds \\=\sE(v_\tau, u_\tau^0, u_\tau^1)
+\int_\tau^t(f,  v)_\EuScript O ds +\int_\tau^t (g, u_s)_\Om ds
\end{multline}
holds
for every $t>\tau$, where the energy functional $\sE$ is defined
by the relation
\begin{equation}\label{en-def}
\sE(v, u, u_t)=E(v,u,u_t)+\int\limits_\Om\Pi(u)dx,
\end{equation}
here
\begin{equation}\label{en-def1}
E(v,u,u_t)=\frac12\left[\mu(t)\|v\|^2_\EuScript O+
\rho(t) \|u_t\|^2_\Om+\| \Delta u\|_\Om^2\right].
\end{equation}
\end{theorem}
\begin{proof}
The proof is quite standard and relies on the method of Galerkin approximations (see e.g.). We place it here for the sake of completeness.  \par
{\em Step 1. Existence.} \par
Let $\{e_i=(e_{1i}, e_{2i})\}_{i\in \N}$ be  the orthonormal basis in $\widetilde X_t=\{v\in X_t: \; (v,n)_\Om=0\}$
consisting of the eigenvectors of the Stokes problem:
\begin{equation}
\label{stokes}
-\De e_i +\nabla p_i =\lambda_i e_i  \quad \mbox{in} \; \EuScript O,~~~
{\rm div}e_i=0, \quad e_i|_{\pd\EuScript O}=0,
\end{equation}
where  $0<\lambda_1\le \lambda_2\le \cdots$ are the corresponding eigenvalues. The existence of solutions to \eqref{stokes} can be shown in the same way as in \cite{temam-NS}.\par
We define the operator 
$N: [\widehat L^2(\Om)]^2 \mapsto [H^{1/2}(\EuScript O)]^2$
by the formula
\begin{equation}\label{fl.n0}
Nu=v ~~\mbox{iff}~~\left\{
\begin{array}{l}
 -\Delta v+\nabla p=0, \quad
   \di v=0 \quad {\rm in}\quad \EuScript O;
\\
 v=0 ~~ {\rm on}~\pd\EuScript O\setminus \Om;
\quad
v=u~~{\rm on} ~ \Om.
\end{array}\right.
\end{equation}
Operator $N$ possesses properties \cite{}
\begin{equation}
\label{n}
N:\, [\widehat H^s_*(\Om)]^2\mapsto [H^{1/2+s}(\EuScript O)]^3\cap X_t
\end{equation}
continuously for every $s\ge -1/2 $ and
\begin{equation}
\label{n1}
\|Nu\|_{1/2+s, \EuScript O}\le C\|u\|^*_{s, \Om},\quad u\in [H^s_*(\Om)]^2.
\end{equation}\par
We also introduce a positive self-adjoint operator $A=\Delta^2$ with the domain $\sD(A)= (H^4\cap \widehat H_0^2)(\Omega)$. It is easy to see that $\sD(A^{1/2})=\widehat H_0^2(\Omega)$.
Denote by $\{g_i\}_{i\in\N}$ the orthonormal basis in $\widehat L_2(\Om)$
which consists of eigenfunctions of the operator $A$
\begin{equation}
\label{101}
Ag_i=\kappa_i g_i
\end{equation}
with the eigenvalues $0<\kappa_1\le\kappa_2\le\ldots$.\par
 Let  $\varphi_i=N g_i$,  where the operator $N$
is defined by (\ref{fl.n0}).\par
We define an approximate solution as a pair of functions $(v_{n,m}; u_n)$:
\begin{equation}
v_{n,m}(t)=\sum_{i=1}^m \alpha_i(t)e_i +\sum_{j=1}^{2n} \dot{\beta}_j(t)\varphi_j, \quad u_n(t)=\sum_{j=1}^{2n}\beta_j(t)g_j \label{approx_sol}
\end{equation}
which satisfy the relations
\begin{equation}
\mu(t)\left(\dot{\alpha}_k(t) +\sum_{j=1}^{2n} \ddot{\beta}_j(t)(\varphi_j,e_k)_\EuScript O\right)+\lambda_k \alpha_k(t)+\sum_{j=1}^{2n} \dot{\beta}_j(t)(\g \varphi_j, \g e_k)_\EuScript O
=(f, e_k)_\EuScript O  \label{e_eq}
\end{equation}
for $k=1,...,m$, and
\begin{multline}
\mu(t)\left(\sum_{i=1}^m \dot{\alpha}_i(t)(e_i, \varphi_k)_\EuScript O+\sum_{j=1}^{2n} \ddot{\beta}_j(t)(\varphi_j, \varphi_k)_\EuScript O\right)+\rho (t)\ddot{\beta}_k(t)\\ +
\sum_{i=1}^m  \alpha_i(t)(\g e_i, \g \varphi_k)_\EuScript O +
\sum_{j=1}^{2n} \dot{\beta}_j(t) (\g e_j, \g \varphi_k)_\EuScript O
+\kappa_k \beta_k(t)+\\
  +(F(u_n(t)), g_k) =
(f(t), \varphi_k)_\EuScript O +(g(t), g_k)_\Om  \label{phi_eq}
\end{multline}
for $k=1,\dots,2n$. 
This system of ordinary differential equations \eqref{e_eq}--\eqref{phi_eq}  is endowed
with the initial data
\[
v_{n,m}(\tau)=\Pi_m(v_\tau-Nu_\tau^1)+NP_n u_\tau^1,
\]
\[
 u_n(\tau)=P_nu_\tau^0, \;\; \dot{u}_n(\tau)=P_nu_\tau^1 ,
\]
 where $\Pi_m$  is  an
 orthoprojector  on $Lin\{e_j: j=1,\ldots,m,\}$ in $\widetilde{X_t}$,  $P_n$
is an orthoprojector on
$Lin\{g_i : i=1,\ldots,n\}$ in $\widehat L_t^2(\Om)$.
Since $\Pi_m$
and $P_n$  are spectral projectors we have that
\begin{equation}\label{id-conv}
(v_{n,m}(\tau); u_n(\tau); \dot{u}_n(\tau))\to  (v_\tau;u_\tau^0;u_\tau^1),\;\text{strongly in }\; H_\tau,\; m,n\to\infty.
\end{equation}
Arguing as in \cite{ChuRyz2011} we infer that system \eqref{e_eq} and \eqref{phi_eq} has a unique  solution
on any time interval $[\tau,T]$.
\par
It follows from (\ref{approx_sol})  that
\[
v_{n,m}(t)=\sum_{i=1}^m \alpha_i(t)e_i + N[\pd_t u_n(t)],
\]
where $N$ is given by (\ref{fl.n0}). This implies
the following boundary compatibility condition
\begin{equation}\label{nm-comp}
v_{n,m}(t)=\pd_t u_n(t)~~ \mbox{on}~~ \Om.
\end{equation}
Multiplying \eqref{e_eq} by $\al_k(t)$ and \eqref{phi_eq} by
$\dot\beta_k(t)$, after summation we obtain an energy relation
of the form \eqref{energy} for the approximate solutions
$(v_{n,m}; u_n)$
(for a similar argument we refer to
 \cite{ChuRyz2011} ).
 Assumptions (A2), (F2), (F3), (G1)  together with  the trace theorem  imply the following a priori estimate:
 \begin{multline}\label{a-pri0}
 \sup_{t\in [\tau,T]}\left[\mu(t)\|v_{n,m}(t)\|_\EuScript O^2 +\rho(t)\| \pd_t u_{n}(t)\|_\Om^2+
 \|\Delta u_n(t)\|^2_\Om\right]\\
+\int_\tau^T\|\nabla v_{n,m}(t)\|_\EuScript O^2 dt + \int_\tau^T \| \pd_t u_{n}(t)\|_{[H_*^{1/2}(\Om)]^2}^2 dt
 \le C(T, \|\cW_\tau\|_{H_\tau}^2)
 \end{multline}
 for any existence interval $[\tau,T]$ of approximate solutions,
 where the constant $C(T, \|\cW_\tau\|_{H_\tau})$ does not depend on $n$ and $m$.
 In particular, this implies that
  any approximate solution  can be extended on any time interval by the standard procedure, i.e., the solution is global.

 It also follows from \eqref{a-pri0} that the sequence  $\{(v_{n,m}; u_n; \pd_t u_n)\}$ contains a subsequence  such that
\begin{align}
&(v_{n,m}; u_n; \pd_t u_n) \rightharpoonup (v; u; \pd_t u) \quad \ast\mbox{-weakly in } L_\infty(\tau,T; H_t),\label{ux-conv}
 \\
&v_{n,m} \rightharpoonup v \quad \mbox{weakly in } L_2(\tau,T;Y).   \label{u_conv}
\end{align}
Moreover, by the Aubin-Dubinsky  theorem
(see, e.g., \cite[Corollary~4]{Simon}) we can assert that
\begin{align}
&u_n \rightarrow u \quad \mbox{strongly in } C(\tau,T; \widehat  H^{2-\epsilon}_0(\Om))
\label{u-strong}
\end{align}
for every $\epsilon>0$. Besides, the  trace theorem yields
\begin{equation}\label{xt-conv}
\pd_t u_n \rightharpoonup \pd_t u \quad \mbox{weakly in } L_2(\tau,T; [H^{1/2}_*(\Om)]^2).
\end{equation}
One can  see  that
$(v_{n,m}; u_n; \pd_t u_n)(t)$ satisfies \eqref{sol_def} with
the test function $\phi$ of the form
\begin{equation}\label{phi-pq}
\phi=\phi_{l,q}=\sum_{i=1}^l \gamma_i(t)e_i +\sum_{j=1}^q\delta_j(t)\varphi_j,
\end{equation}
where $l\le m$, $q\le n$ and $\gamma_i$, $\delta_j$
are scalar absolutely continuous functions on $[\tau,T]$
such that $\dot{\ga}_i,\dot{\delta}_j\in L_2(\tau,T)$ and $\gamma_i(T)=\delta_j(T)=0$. Thus using (\ref{ux-conv})--
(\ref{u_conv}) we can pass to the limit and
show that $(v; u; \pd_t u)(t)$ satisfies \eqref{e_eq}--\eqref{phi_eq}
with  $\phi=\phi_{l,q}$, where $l$ and $q$ are arbitrary.
By (\ref{id-conv}) and (\ref{u-strong}) we have $\cW(\tau)=\cW_\tau$.
Compatibility condition \eqref{com} follows from  (\ref{nm-comp})
and  (\ref{xt-conv}).
\par
To conclude the proof of the existence of weak solutions
we only need to show that any function $\psi$ in $\EuScript L_T^0$ can be approximated by
a sequence of functions of the form (\ref{phi-pq}). This can be done in the following way. We first approximate the corresponding boundary value of $b$
by a finite linear combination $h$ of $\xi_j$, then we  approximate the difference $\psi-Nh$ (with $N$ define by (\ref{fl.n0}))
by finite linear combination of $e_i$. Limit transition in nonlinear terms is quite standard, so we omit it here.
Thus the existence of weak solutions is proved.\par
{\em Step 2. Energy equality.} \par
To prove the energy equality for a weak solution we follow the scheme
presented in  ~\cite{KochLa_2002}. 
We introduce a finite difference operator $D_h$, depending on a small parameter $h$.
Let $g$ be a bounded function on $[\tau,T]$ with values in some Hilbert space. We extend $g(t)$ for all $t\in \R$ by defining $g(t)=g(0)$ for $t<\tau$ and $g(t)=g(T)$ for $t>T$. With this extension we denote
\[
g^+_h(t)=g(t+h)-g(t), \quad g^-_h(t)=g(t)-g(t-h), \quad
D_h g(t)=\frac 1{2h}(g^+_h(t)+ g^-_h(t)).
\]
Properties of the operator $D_h$ are collected in Proposition 4.3 \cite{KochLa_2002}.
\par
Taking in (\ref{sol_def})
 $\phi(t)=\int_t^T\chi(s)ds\cdot \phi$, where $\chi$
is a smooth scalar function  and $\phi$ belongs to the space
\begin{equation}\label{space-W}
\widehat{Y}=\left\{
\phi\in  Y \left|  \;
  \phi|_\Om=
b \in \widehat H^2_0(\Om)  \right. \right\},
\end{equation}
one can see that the weak solution $(v(t); u(t))$  satisfies the relation
       \begin{multline}
       \label{wc}
          \mu(t) (v(t), \phi)_{\EuScript O}
+\rho(t)(u_t(t), b)_\Om
 = (v_\tau, \phi)_{\EuScript O} + (u_\tau^1, b)_\Om+
\int_\tau^t\big[
\frac12\mu'(s)(v, \phi)_{\EuScript O}\\+\frac12\rho'(s)(u_t, b)_{\Om}-(\g v, \g \phi)_{\EuScript O}-(\Delta u, \Delta b)_\Om +(F(u), b)_\Om+(f, \phi)_{\EuScript O}+(g, b)_\Om\big] ds
        \end{multline}
for  all $t\in [\tau,T]$ and
$\phi\in \widehat{Y}$  with $\phi\big\vert_\Om=b$.\par
 The vector $(v(t), u(t), u_t(t))$ is weakly continuous in $H_t$  for any weak solution  $(v(t), u(t))$ to problem \eqref{fl.1}--\eqref{IC}.
Indeed, it follows from (\ref{wc}) that
 $(v(t), u(t))$  satisfies the relation
       \[
            \mu(t)(v(t),\phi)_\EuScript O
 = \mu(\tau) (v_\tau, \phi)_{\EuScript O}
+\int_\tau^t\left[\frac12\mu'(s)(v, \phi)_{\EuScript O} -
  (\g v,\g \phi)_{\EuScript O}
          +  (f(s), \phi)_{\EuScript O}  \right] ds
        \]
for almost all $t\in [\tau,T]$ and for all
$\phi\in Y_0=\{ v\in Y :\, v|_\Om=0\}\subset\widehat {Y}\subset Y$,
where $\widehat{Y}$ is given by \eqref{space-W}.
This implies that $v(t)$ is weakly continuous in $Y_0'$.
Since $X_t\subset Y'_0$, for any $\tau<t<T$ we can apply  Lions lemma
(see \cite[Lemma 8.1]{LiMa_1968}) and conclude that $v(t)$ is
weakly continuous in $X_t$. The same lemma  gives us weak continuity
 of $u(t)$ in $\widehat H_0^2(\Om)$. Now using  (\ref{wc}) again with $\phi\in \widehat{Y}$
 we conclude that
$
 t\mapsto  (u_t(t), b)_\Om
$ is continuous
 for every $b\in \widehat  H_0^2(\Om)$. This implies that
 $
 t\mapsto  u_t(t)$ is weakly continuous
 in
$[L_2(\Om)]^2 $.
 Using weak continuity of weak solutions,
we can extend the variational relation in \eqref{sol_def} on the class of test functions
from $\EuScript L_T$ (instead of $\EuScript L_T^0$) by an appropriate limit transition. More precisely, one can show that
 any weak solution $(v; u)$  satisfies the relation
 \begin{multline}
\label{sol_def_t}
-\int\limits_\tau^T\mu(t)(v,\phi_t)_{\EuScript O}dt+\int\limits_\tau^T (\g v, \g \phi)_{\EuScript O}
dt-\int\limits_\tau^T\rho(t) (u_t, b_{t})_\Om
dt+\int\limits_\tau^T(F(u), b)  dt\\+\int\limits_\tau^T (
\Delta u, \Delta b)_\Om dt=(v_\tau,
\phi(\tau))_{\EuScript O}+(u_\tau^1, b(\tau))_\Om-\mu(T)(v(T),
\phi(T))_{\EuScript O}\\-\rho(T)(u_T(T), b(T))_\Om+\frac12\int\limits_\tau^T\mu'(t)(v,\phi)_{\EuScript O}dt+\frac12\int\limits_\tau^T\rho'(t) (u_t, b)_\Om
dt
\\+\int\limits_\tau^T(f, \phi)_{\EuScript O}
dt+\int\limits_\tau^T(g, b)_{\Omega}
dt,
\end{multline}
    for every  $\psi=(\phi, b)\in\EuScript L_T$ .\par Let $(v(t), (t))$ be a weak solution to problem  \eqref{fl.1}--\eqref{IC}.
Now  we use
\begin{equation}
\label{98}
\phi=\frac 1{2h}\int_{t-h}^{t+h} v(s) ds
\end{equation}
as a test function in \eqref{sol_def_t}. For the shell component we have test function $b=\phi|_\Om=D_h u$ -- the same one that used in \cite{KochLa_2002} for the full Karman model. \par
Arguing as in the proof of  Proposition 4.3 \cite{KochLa_2002} we can infer
\begin{multline}
\label{l1}
\lim\limits_{h\to 0} \left(\int\limits_\tau^T\mu(t)(v(t), D_h v(t))_{\EuScript O}dt-\frac12\int\limits_\tau^T\mu'(t)(v(t), \int\limits_{t-h}^{t+h} v(s)ds)_{\EuScript O} dt\right)\\=\frac12\left(\mu(T)\|v(T)\|_{\EuScript O}^2-\mu(\tau)\|v(\tau)\|_{\EuScript O}^2\right)
\end{multline}
\begin{multline}
\label{l2}
\lim\limits_{h\to 0} \left(\int\limits_\tau^T\rho(t)(u_{t}(t), D_h u_t(t))_\Om dt-\frac12\int\limits_\tau^T\rho'(t)(u_t(t),  D_h u(t))_\Om dt\right)\\=\frac12\left(\rho(T)\|u_T(T)\|_{\Om}^2-\rho(\tau)\|u_\tau(\tau)\|_{\Om}^2\right)
\end{multline}
Then, relying on \eqref{sol_def_t}, \eqref{l1}, and  \eqref{l2} we can conclude the proof. 
All the arguments for the fluid component in our model are the same as in \cite{cr-full-karman}, and the arguments for the plate component  are analogous to those presented in the proof of  Lemma 4.1 \cite{KochLa_2002}. 
  This makes it possible to prove the energy equality in \eqref{energy}.\par
Continuity of weak solutions with respect to $t$
can be obtained in the standard
way from the energy equality and weak continuity (see \cite[Ch.~3]{LiMa_1968}
and also \cite{KochLa_2002}).\par
{\em Step 3. Continuity with respect to the initial data and uniqueness.} \par
It follows from energy estimate \eqref{energy} and (F3) that if $\|W_\tau\|_{H_\tau}\le R$, then there exists $C(R)>0$ such that $\|U(t, \tau)W_\tau\|_{H_t}\le C(R)$.   Consequently,  the Gronwall  lemma and (F1) yield estimate \eqref{contin}. The uniqueness of solutions follows.
\end{proof}
Now we are in position to show the existence of a time-dependent absorbing family.
\begin{lemma}
\label{lem}
Let $t \ge \tau$ . Let $U(t, \tau )W_\tau$ be the solution of \eqref{fl.1}--\eqref{IC} with initial time
$\tau$ and initial data  $W_\tau \in H_\tau$. Then, if (F4) holds, there exist $\omega>0$,
$K \ge 0$ and an increasing positive function $Q$ such that
\begin{equation}
\label{dis}
\|U(t, \tau )W_\tau\|_{H_t} \le Q(\|W_\tau\|_{H_\tau})e^{-ω(t-\tau)} + K,\;\; \tau \le t.\end{equation}
\end{lemma}
\begin{proof}
We constuct the Lyapunov functional of the form
\begin{equation}
\label{lap}
L(t)=\EuScript E(t)+\delta\left(\mu(t)(v, Nu)_{\EuScript O}+\rho(t)(u_t, u)_\Om\right).
\end{equation}
It is easy to see from (F3) and the properties of the operator $N$ that there exist $c_i>0$, $i=\overline{1,4}$ such that
\begin{equation}
\label{est}
-c_1+c_2E(t)\le L(t)\le c_3 \EuScript E(t)+c_4, 
\end{equation}
where $E(t)$ is defined by \eqref{en-def1}.  
All the calculations below can be performed on Galerkin approximations. It follows from the energy inequality \eqref{energy} and (F4)
that
\begin{multline*}\frac{d}{dt}L(t)=-\|\g v\|_{\EuScript O}^2 +\rho'(t)\|u_t\|_\Om^2+\mu'(t)\|v\|_{\EuScript O}^2+(f, v)_{\EuScript O}+(g, u_t)_{\Om}+\delta\mu(t)(v, Nu_t)_{\EuScript O}\\
-\delta\|\Delta u\|_\Om^2+\delta \rho(t) \|u_t\|_\Om^2 +\delta \rho'(t) (u_t, u)_\Om+\delta \mu'(t) (v, Nu)_{\EuScript O}\\-\delta (F(u), u)_\Om+\delta(f, Nu)_{\EuScript O}+\delta (g, u)_\Om-\delta (\g v, \g Nu)_{\EuScript O}\le -\omega L(t)+C(\|f\|_{\EuScript O}^2+\|g\|_{-1/2,  \Om}^2)
\end{multline*}
for some $\omega, C>0$. Consequently,
\[\frac{d}{dt}L(t)+\omega L(t)\le C(\|f\|_{\EuScript O}^2+\|g\|_{-1/2,  \Om}^2)\]
and using the Gronwall lemma, (G2), (F2), and \eqref{est} we come to \eqref{dis}.
\end{proof}
Lemma \ref{lem} yields the existence of a time-dependent absorbing family with the entering time $\Theta=\max\{0, \frac 1\omega log\frac{Q(R)}{1+K}\}$. 
\section{Pullback attractor.}
In order to establish the existence of a pullback attractor to the system considered, our remaining task  is to show estimate \eqref{ac}.
\begin{lemma}
Let $W^i(t)=(v^i(t), u^i(t), u_{t}^i(t))$, $i=\overline\{1,2\}$ be two weak solutions to problem \eqref{fl.1}--\eqref{IC} with initial conditions $W_\tau^i\in H_{t-T_0}$,  $\|W^i\|_{H_{t-T_0}}\le R$.   Then, for any $\varepsilon>0$ there exists $T_0>0$ and a positive constant $C(t, T_0)$ such that 
\begin{equation}
\label{est1}
\|W^1(t)-W^2(t)\|_{H_{t}}\le \varepsilon +C(T_0, R)\max\limits_{[t-T_0, t]}(\|u^1(s)-u^2(s)\|_{2-\epsilon, \Om}^2), 
\end{equation}
for any $\epsilon>0$. 
\end{lemma}
\begin{proof}
We use the notations $W(t)=(v(t), u(t), u_t(t))=W^1(t)-W^2(t)$.
It follows from the energy inequality that 
\begin{equation}
\label{91}
\frac{d}{d\xi} E(\xi)\le -\|\g v\|_{\EuScript O}^2+\mu'(\xi)\|v\|_{\EuScript O}^2+\rho'(\xi)\|u_\xi\|_\Om^2+(F(u^1)-F(u^1), u_\xi)_\Om.
\end{equation}
Integrating \eqref{91} over the interval $[s, t]$ and then $[t-T_0, t]$ we come to
\begin{multline}
\label{92}
T_0 E(t)\le \int\limits_{t-T_0}^tE(s)ds-\int\limits_{t-T_0}^t \int\limits_{s}^t\|\g v\|_{\EuScript O}^2d\xi ds+\int\limits_{t-T_0}^t \int\limits_{s}^t\mu'(\xi)\|v\|_{\EuScript O}^2d\xi ds\\+\int\limits_{t-T_0}^t \int\limits_{s}^t\rho'(\xi)\|u_\xi\|_\Om^2d\xi ds+\int\limits_{t-T_0}^t \int\limits_{s}^t(F(u^1)-F(u^1), u_\xi)_\Om d\xi ds.
\end{multline}
It follows from the trace theorem and assumption (F1) that for any $\sigma>0$
\begin{multline}
\label{93}
\left|\;\int\limits_{t-T_0}^t \int\limits_{s}^t(F(u^1)-F(u^1), u_\xi)_\Om d\xi ds\right|\le \sigma \int\limits_{t-T_0}^t \int\limits_{s}^t\|\g v\|_{\EuScript O}^2d\xi ds\\+C(T_0, R, \sigma)\max\limits_{[t-T_0, t]}\|u\|_{2-\epsilon, \Om}^2.
\end{multline}
Integrating \eqref{91} over the interval $[t-T_0, t]$ and taking into consideration (A2) and (F1) we obtain 
\begin{multline}
\label{94}
E(t)+\int\limits_{t-T_0}^t\|\g v\|_{\EuScript O}^2ds\le C(R)\int\limits_{t-T_0}^t\|u\|_{2-\epsilon,\Om}^2ds+E(t-T_0)\le C(T_0, R)\max\limits_{[t-T_0, t]}\|u\|_{2-\epsilon, \Om}^2+C(R).
\end{multline}
It is a straightforward consequence of the trace theorem that 
\begin{equation}
\label{95}
\int\limits_{t-T_0}^t E(s)ds\le  C\int\limits_{t-T_0}^t\|\g v\|_{\EuScript O}^2ds+\int\limits_{t-T_0}^t\|u\|_{2, \Om}^2ds.
\end{equation}
Now we estimate the last term in \eqref{95}.  Substituting into \eqref{sol_def} $b=u$ and $\phi=Nu$ and choosing $\tau=t-T_0$ and $T=t$ we arrive at 
\begin{multline*}
\int\limits_{t-T_0}^t\|u\|_{2, \Om}^2ds\le \int\limits_{t-T_0}^t\rho(s)\|u_s\|_{2, \Om}^2ds-\rho(t)(u_t(t), u(t))_\Om+ \rho(t-T_0)(u_t(t-T_0), u(t-T_0))_\Om\\+\int\limits_{t-T_0}^t(F(u^1)-F(u^2), u_s)_\Om ds+\int\limits_{t-T_0}^t\mu(s)(v, Nu_s)_{{\EuScript O}} ds -\int\limits_{t-T_0}^t(\g v, \g Nu)_{{\EuScript O}} ds\\- \mu(t)(v(t), Nu(t))_{{\EuScript O}}+ \mu(t-T_0)(v(t-T_0), Nu(t-T_0))_{{\EuScript O}}+\int\limits_{t-T_0}^t\mu'(s)(v, Nu)_{{\EuScript O}} ds.
\end{multline*}
Relying on the properties of the operator $N$, the trace theorem, and  (A3) we have the estimate
\begin{equation}
\label{96}
\int\limits_{t-T_0}^t\|u\|_{2, \Om}^2ds\le
C(T_0, R)\max\limits_{[t-T_0, t]}\|u\|_{2-\epsilon, \Om}^2+С\int\limits_{t-T_0}^t\|\g v\|_{\EuScript O}^2ds+C(R).
\end{equation}
Combining \eqref{92}--\eqref{96} and choosing $\sigma$ in \eqref{93} we obtain
\[E(t)\le \frac{C(R)}{T_0}+C(T_0, R)\max\limits_{[t-T_0, t]}\|u\|_{2-\epsilon, \Om}^2,\]
which leads immediately to the assertion of the lemma.\\
Now we formulate our main result.
\end{proof}
\begin{theorem}
The process $U(t, \tau)$ generated by problem \eqref{fl.1}--\eqref{IC} posessess a pullback attractor. 
\end{theorem}
\begin{proof}
There exists a time-dependent absorbing family and we have in hand Lemma 2. Therefore, to use Theorem 1 it remains to show that \eqref{p} holds true for $\Phi_{T_0,t}(W^1,W^2)=C(T_0, R)\max\limits_{[t-T_0, t]}\|u^1(s)-u^2(s)\|_{2-\epsilon, \Om}^2$. Let $W^n(t)=(v^n(t), u^n(t), u_t^n(t))$ be a sequence of solutions to problem 
\eqref{fl.1}--\eqref{IC} corresponding to initial data $W^n$,  from $D_{t-T_0}$, i.e. $\|W^n\|_{H_{t-T_0}}\le R$  .Then, it follows from Lemma 1 that up to a subsequence
$$
\begin{array}{l}
u^n(s)-u^m(s) \to  0,\;\; \text{weak-* in}\;\;  L_\infty(t-T_0, t; \widehat H_0^2(\Om)),\\
u_s^n(s)-u_s^m(s) \to  0 ,\;\;\text{weakly in} \;\; L_2(t-T_0, t; H_*^{1/2}(\Om)).
\end{array}
$$
By the Aubin’s compactness lemma \cite{Simon}, we have \eqref{p}. This together with Theorem 1 completes the proof.
\end{proof}

\textbf{Acknowledgements.}\par
The author is grateful to Irina Kmit, Humboldt University of Berlin, for fruitful discussion and valuable comments.\par
The author was partially supported by the Volkswagen Foundation grant within frameworks of the international
project ”Modeling, Analysis, and Approximation Theory toward Applications in Tomography and
Inverse Problems.”

\medskip

Received xxxx 20xx; revised xxxx 20xx.
\medskip
\end{document}